\documentclass[12pt]{amsart}
\usepackage{amsmath,amsfonts,amssymb,amsthm,amstext,pgf,graphicx,hyperref,verbatim,lmodern,ytableau,textcomp,color,young,youngtab,tikz}
\usepackage[mathscr]{euscript}
\theoremstyle{plain}
\newtheorem{thm}{Theorem}[section]
\newtheorem{lem}[thm]{Lemma}
\newtheorem{prop}[thm]{Proposition}
\newtheorem{cor}[thm]{Corollary}
\newtheorem{rem}[thm]{Remark}

\newtheorem{exmp}[thm]{Example}

\theoremstyle{definition}
\newtheorem{defn}{Definition}[section]

\DeclareMathOperator{\1}{id}
\DeclareMathOperator{\Ustd}{UStd(\lambda)}
\DeclareMathOperator{\Ustdexample}{UStd}
\DeclareMathOperator{\Lstd}{LStd(\lambda)}
\DeclareMathOperator{\Ustdc}{UStd(\lambda^{\prime})}
\DeclareMathOperator{\std}{Std(\lambda)}

\DeclareMathOperator{\Par}{Par(n)}
\DeclareMathOperator{\Cpar}{CPar(n)}
\DeclareMathOperator{\Ncpar}{NCPar(n)}
\title{}
\begin{document}
\title
{Total Variation Cutoff for the Transpose top-$2$ with random shuffle}
\author{Subhajit Ghosh}
\address[Subhajit Ghosh]{Department of Mathematics, Indian Institute of Science, Bangalore 560 012}
\email{gsubhajit@iisc.ac.in}
%\email{subhajit1071993@gmail.com}

%---------ABSTRACT----------
\begin{abstract}
In this paper, we investigate the properties of a random walk on the alternating group $A_n$ generated by $3$-cycles of the form $(i,n-1,n)$ and $(i,n,n-1)$. We call this the transpose top-$2$ with random shuffle. We find the spectrum of the transition matrix of this shuffle. We show that the mixing time is of order $\left(n-\frac{3}{2}\right)\log n$ and prove that there is a total variation cutoff for this shuffle.
\end{abstract}
\maketitle
%\tableofcontents
%------------Introduction----------
\section{Introduction}\label{intro}
The convergence of random walks on finite groups is a well-studied topic in probability theory. Under certain natural conditions, a random walk converges to a unique stationary distribution. The topic of interest, in this case, is the mixing time i.e., the minimum number of steps the random walk takes to get close to its stationary distribution in the sense of total variation distance. To understand the convergence of random walks, it is helpful to know the eigenvalues and eigenvectors of the transition matrix. There is a large body of literature dealing with the convergence of random walks on the symmetric group. Random walks on the symmetric group can be practically described by shuffling of a deck of cards. Arrangements of this deck can be described by elements of the symmetric group. There are various types of shuffling algorithms of a deck of $n$-cards labelled from $1$ to $n$. Shuffling problems on $S_n$ can be generalised to other finite groups. One can take a look at \cite{FOW,Diaconis Note1,Diaconis Note2,Saloff-Coste Note,DS,LP,RS} for various shuffles.

This work is mainly motivated by the work of Flatto, Odlyzko and Wales \cite{FOW}, where they studied the \emph{transpose top with random shuffle on $S_n$} (interchanging the top card with a random card). The eigenvalues for the transition matrix of the transpose top with random shuffle were obtained from the \emph{Jucys-Murphy} elements for the symmetric group \cite{VO}. We consider an analogous variant called the \emph{transpose top-$2$ with random shuffle} on the alternating group. In this work, we will prove total variation cutoff for transpose top-$2$ with random shuffle on $A_n$. The transpose top-$2$ with random shuffle on $A_n$ is a random walk on $A_n$ driven by a probability measure $P$ on $A_n$ defined as follows:
\begin{equation}\label{def of P}
P(\pi)=
\begin{cases}
\frac{1}{2n-3},\text{ if }\pi=\1,\\
\frac{1}{2n-3},\text{ if }\pi=(i,n-1,n)\text{ for }i\in\{1,\dots,n-2\},\\
\frac{1}{2n-3},\text{ if }\pi=(i,n,n-1)\text{ for }i\in\{1,\dots,n-2\},\\
0,\text{ otherwise}.
\end{cases}
\end{equation}
The distribution after $k$ transitions is given by $P^{*k}$ (convolution of $P$ with itself $k$ times, details explanation will be given later in this section). Before stating the main result of this paper we first define the \emph{total variation distance} between two probability measures.
\begin{defn}
Let $\mu$ and $\nu$ be two probability measures on $\Omega$. The \emph{total variation distance} between $\mu$ and $\nu$ is defined to be \[||\mu-\nu||_{\text{TV}}:=\sup_{A\subset\Omega}|\mu(A)-\nu(A)|.\]
\end{defn}
The total variation distance between $\mu$ and $\nu$ can also be written as $||\mu-\nu||_{\text{TV}}=\frac{1}{2}\sum\limits_{x\in\Omega}|\mu(x)-\nu(x)|$ (see \cite[p. 48, proposition 4.2]{LPW}).
\begin{thm}\label{thm:UB for Alt}
Let $U_{A_n}$ denote the uniform distribution on $A_n$. For the transpose top-$2$ with random shuffle on $A_n$ driven by $P$, we have the following:
\begin{enumerate}
\item $||P^{*k}-U_{A_n}||_{\text{TV}}<\frac{1}{\sqrt{2}}e^{-c}+o(1)$, for $k\geq(n-\frac{3}{2})(\log n+c)$ and $c>0$.
\item $\lim\limits_{n\rightarrow\infty}||P^{*k_n}-U_{A_n}||_{\text{TV}}=0$, for any $\epsilon\in (0,1)$ and $k_n=\lfloor(1+\epsilon)(n-\frac{3}{2})\log n\rfloor.$
\end{enumerate}
\end{thm}
The proof of this theorem will be given in Section \ref{sec:upper bound}. We first review some concepts and terminologies before discussing the objective of this paper, which will be frequently used in this paper.
\subsection{Preliminaries}
Let $G$ be a finite group. A \emph{linear representation} $\rho$ of $G$ is a homomorphism from $G$ to $GL(V)$, where $V$ is a finite-dimensional vector space over $\mathbb{C}$ and $GL(V)$ is the set of all invertible linear maps from $V$ to itself. Let $\mathbb{C}[G]$ be the set of all formal linear combinations of the elements of $G$ with complex coefficients. In particular, if we take $V=\mathbb{C}[G]$, then the \emph{right regular representation} $R:G\longrightarrow GL(V)$ is defined by $g\mapsto\left(\sum_{h\in G}C_hh\mapsto\sum_{h\in G}C_hhg\right),\;C_h\in\mathbb{C}$. i.e., $R(g)$ is an invertible matrix over $\mathbb{C}$ of order $|G|\times|G|$. The \emph{dimension} $d_{\rho}$ of the representation is defined to be the dimension of the vector space $V$. The \emph{character} $\chi^{\rho}(g)$ (also denoted by $\psi^{\rho}(g)$) of this representation is defined to be the trace of the matrix $\rho(g)$. The representation $\rho$ is said to be \emph{irreducible} if there does not exist any nontrivial proper subspace $W$ of $V$ such that $\rho(g)\left(W\right)\subset W$ for all $g$ in $G$ (i.e., $W$ is invariant under $\rho(g),\;\text{for all}\; g\in G$). Two representations $\rho_1:G\longrightarrow GL(V_1)$ and $\rho_2:G\longrightarrow GL(V_2)$ of $G$ are said to be \emph{isomorphic} if there exists an invertible linear transformation $T:V_1\longrightarrow V_2$ such that $T\circ\rho_1(g)=\rho_2(g)\circ T$ for all $g\in G$. %For a linear representation $\rho:G\longrightarrow GL(V)$, w
We define the \emph{restriction} $\rho\downarrow^G_H$ of $\rho$ to a subgroup $H$ of $G$ by $\rho\downarrow^G_H(h)=\rho(h)$ for all $h\in H$. If $\chi^{\rho}$ is the character of $\rho$, then the character of the restriction 
$\rho\downarrow^G_H$ is denoted by $\chi^\rho\downarrow^G_H$. We will state some results from representation theory of finite groups without proof for more details see \cite{Sagan,Serre,Amri}.

Let $G$ be a finite group and $\;p,\;q$ be two probability measures on $G$. The \emph{convolution $p*q$ of $p$ and $q$} is defined by $(p*q)(x):=\sum_{y\in G}p(xy^{-1})q(y)$. Given a probability measure $p$ on $G$, we define the \emph{Fourier transform} $\widehat{p}$ of $p$ at the representation $\rho$ by the matrix $\sum_{x\in G}p(x)\rho(x)$. One can easily check that $\widehat{(p*q)}(\rho)=\widehat{p}(\rho)\widehat{q}(\rho)$ for two probability measures $p,\;q$ on $G$. In particular, for the right regular representation $R$, the matrix $\widehat{p}(R)$ can be thought of as the action of the group algebra element $\sum_{g\in G}p(g)g$ on $\mathbb{C}[G]$ from the right. 

A \emph{random walk on a finite group $G$ driven by a probability measure $p$} is a Markov chain with state space $G$ and transition probabilities $M_p(x,y)=\;p(x^{-1}y)$, $x,y\in G$. In terms of the Fourier transform defined above, the transition matrix $M_p$ is the transpose of $\widehat{p}(R)$. Also, the distribution after $k^{th}$ transition will be $p^{*k}$ (convolution of $p$ with itself $k$ times) i.e., the probability of getting into state $y$ from state $x$ through $k$ transitions is $p^{*k}(x^{-1}y)$. A Markov chain (discrete time, finite state space) is said to be \emph{irreducible} if it is possible for the chain to reach any state starting from any state using only transitions of positive probability. Therefore a random walk on a finite group $G$ driven by a probability measure $p$ is irreducible if for any two elements $x,y\in G$, there exists $t$ (depending on $x$ and $y$) such that $p^{*t}(x^{-1}y)>0$. We note that \emph{the random walk is irreducible if the support of $p$ generates $G$}. To see this, let $\Gamma$ be the support of $p$. Then from the hypothesis $\Gamma$ generates $G$. Therefore given any two arbitrary elements $x,y\in G$, $x^{-1}y$ can be written as $\gamma_1\gamma_2\dots\gamma_t$ for $\gamma_1,\gamma_2,\dots,\gamma_t\in\Gamma$. Thus $p^{*t}(x^{-1}y)\geq p(\gamma_1)p(\gamma_2)\dots p(\gamma_t)>0$ and the assertion has been proved. A probability distribution $\Pi$ is said to be a \emph{stationary distribution} of a Markov chain with transition matrix $M$ if $\Pi M=\Pi$ (i.e., if the initial distribution for the chain is $\Pi$, then the distribution after one transition, and hence after any number of transitions, will remain $\Pi$). Any irreducible Markov chain possesses a unique stationary distribution. In case of an irreducible random walk on a finite group $G$ driven by a probability measure $p$, \emph{the stationary distribution happens to be the uniform distribution on $G$} (since $\sum_{x\in G}M_p(x,y)=\sum_{x\in G}p(x^{-1}y)=\sum_{z\in G}p(z)=1,\;z=x^{-1}y$ for all $y\in G$). We note that the irreducible random walk on $G$ driven by $p$ is \emph{time reversible} if and only if $M_p(x,y)=M_p(y,x)$ for all $x,y\in G$ i.e., if and only if $p(x)=p(x^{-1})$ for all $x$ in $G$. Given a discrete time Markov chain with finite state space, we define the \emph{period} of a state $x$ by the greatest common divisor of the set of all times when it is possible for the chain to return to the starting state $x$. The period of all the states of an irreducible Markov chain are same (see \cite[p. 8, lemma 1.6]{LPW}). An irreducible Markov chain is said to be \emph{aperiodic} if the common period for all its states is $1$.

Let $G$ be a finite group. From now on, we denote the uniform distribution on $G$ by $U_G$. We have just seen that the unique stationary measure for an irreducible random walk on a finite group $G$ driven by a probability measure $p$ is the uniform measure on $G$. Now if the random walk is aperiodic, then the distribution after $k^{\text{th}}$ transition converges to $U_{G}$ in total variation distance as $k\rightarrow\infty$. We note that, given an irreducible and aperiodic random walk on a finite group $G$ driven by some probability measure $p$, the total variation distance $||p^{*k}-U_{G}||_{TV}$ decreases as $k$ increases.

We now define the total variation cutoff phenomenon. Let $\{G_n\}_{0}^{\infty}$ be a sequence of finite groups. % and $U_{G_n}$ be the uniform measure on $G_n$ for all $n$. 
For each $n\geq 0$, consider the irreducible and aperiodic random walk on $G_n$ driven by some probability measure $p_n$ on $G_n$. We say that the \emph{total variation cutoff phenomenon} holds for the family $\{(G_n,p_n)\}_0^{\infty}$ if there exists a sequence $\{\tau_n\}_0^{\infty}$ of positive reals such that the following hold:
\begin{enumerate}
\item $\lim\limits_{n\rightarrow\infty}\tau_n=\infty,$
\item For any $\epsilon\in (0,1)$ and $k_n=\lfloor(1+\epsilon)\tau_n\rfloor$, $\lim\limits_{n\rightarrow\infty}||p_n^{*k_{n}}-U_{G_n}||_{TV}=0$ and
\item For any $\epsilon\in (0,1)$ and $k_n=\lfloor(1-\epsilon)\tau_n\rfloor$, $\lim\limits_{n\rightarrow\infty}||p_n^{*k_{n}}-U_{G_n}||_{TV}=1$.
\end{enumerate}
Here $\lfloor x\rfloor$ denotes the floor  of $x$ (the largest integer less than or equal to $x$). Informally, we will say that $\{(G_n,p_n)\}_0^{\infty}$ has a total variation cutoff at time $\tau_n$. The cutoff phenomenon depends on the multiplicity of the second largest eigenvalue of the transition matrix, for more details, see \cite{Cut-off Diaconis}.
%Informal description
Informally, this says that there exist real numbers $\tau_n$ and neighbourhoods $I_n$ of $\tau_n$ such that $||p_n^{*k_{n}}-U_{G_n}||_{TV}$ decreases drastically to a very small positive quantity whenever $k_n$ lies in $I_n$, for large $n$. 

\subsection{Transpose Top-$2$ With Random Shuffle on $A_n$}
Let $S_n$ be the set of all permutations of the numbers $ 1,2,\dots,n $. The set $S_n$ forms a group under the multiplication of permutations, known as \emph{symmetric group}. A permutation in $S_n$ is said to be an \emph{even permutation} if it can be expressed as a product of even number of transpositions (not necessarily disjoint). The set of all even permutations in $S_n$ forms a subgroup of the symmetric group, known as the \emph{alternating group} and is denoted by $A_n$. Now we will define the transpose top-$2$ with random shuffle on $A_n$. Suppose we have a deck of cards labelled from $1$ to $n$ such that the arrangement of the deck is a permutation in $A_n$. Then the \emph{transpose top-$2$ with random shuffle on $A_n$} is a lazy variant of the following: First transpose the top two cards, then choose one of them and interchange it with a card randomly chosen from the remaining $(n-2)$ cards. More formally, this can be described as a random walk on $A_n$ driven by $P$. i.e., any permutation in $A_n$ can either go to itself or be multiplied on the right by a $3$-cycle of the form $(i,n-1,n)$ or $(i,n,n-1)$  with probability $\frac{1}{2n-3}$.
\begin{prop}\label{irreducibility and aperiodicity}
The Markov chain formulated in \eqref{def of P} from the transpose top-$2$ with random shuffle on $A_n$ is irreducible and aperiodic.
\end{prop}
\begin{proof}
We know that the $3$-cycles generate $A_n$. Let $a,b,c$ be any three distinct integers from $\{1,2,\dots,n\}$. If none of $a,b,c$ is $(n-1)$ or $n$, we have, 
\begin{align*}
(a,b,c)&=(c,n,n-1)(a,b,n-1)(c,n-1,n)\\
&=(c,n,n-1)(b,n,n-1)(a,n-1,n)(b,n-1,n)(c,n-1,n).
\end{align*}
If one of $a,b,c$ is $(n-1)$ or $n$, without loss of any generality we may assume $c$ is either $(n-1)$ or $n$ and we have the following,
\begin{align*}
&\quad\;(a,b,n-1)=(b,n,n-1)(a,n-1,n)(b,n-1,n), \text{ and }\\
&\quad\;(a,b,n)=(b,n-1,n)(a,n,n-1)(b,n,n-1).
\end{align*}
If any two of $a,b,c$ are $(n-1)$ and $n$, then $(a,b,c)$ takes the form $(\;\cdot\;,n-1,n)$ or $(\;\cdot\;,n,n-1)$. Therefore the support of the measure $P$ generates $A_n$ and hence the chain is irreducible. Given any $\pi\in A_n$, the set of all times when it is possible for the chain to return to the starting state $\pi$ contains the integer $1$ ($\because P(\1)\neq 0$). Therefore the period of the state $\pi$ is $1$ and hence from irreducibility all the states of this chain have period $1$. Thus this chain is aperiodic.
\end{proof}
Proposition \ref{irreducibility and aperiodicity} gives an affirmative answer to the question of existence and uniqueness of stationary distribution. It also says that after a large number of transitions the distribution of the chain behaves like the stationary distribution. We know, in the case of an irreducible random walk on a finite group driven by some probability measure, that the unique stationary distribution is the uniform distribution on that group. More precisely, the distribution after $k^{\text{th}}$ transition will converge to $U_{A_n}$ as $k\rightarrow\infty$. 

In Section \ref{sec:representation} we will find the spectrum of $\widehat{P}(R)$. We will prove Theorem \ref{thm:UB for Alt} in Section \ref{sec:upper bound}. In Section \ref{sec:lower bound}, we will find a lower bound of $||P^{*k}-U_{A_n}||_\text{{TV}}$ for $k=(n-\frac{3}{2})(\log n+c),\;c\ll 0$ (large negative). At the end of this paper, we will show that the cutoff occurs at $(n-\frac{3}{2})\log n$. 

\subsection*{Acknowledgement} I would like to thank my advisor Arvind Ayyer for fruitful conversations and suggestions in the preparation of this paper. I would also like to thank Amritanshu Prasad, Arvind Ayyer and Pooja Singla for proposing the problem. I am also thankful to all the reviewers for the careful reading of the manuscript and their valuable comments. I would like to acknowledge support in part by a UGC Centre for Advanced Study grant.
%------------------EIGENVALUES OF $\widehat{P}(R)$-------------
\section{Spectrum of $\widehat{P}(R)$}\label{sec:representation}
Let $\mathcal{P}=\1+\sum\limits_{i=1}^{n-2}\left((i,n-1,n)+(i,n,n-1)\right)\in \mathbb{C}[A_n]$. Recall that $\widehat{P}(R)$ is the Fourier transform of the probability measure $P$ at right regular representation $R$ of $A_n$, which can be written as $\widehat{P}(R)=\frac{1}{2n-3}\mathcal{P}$. Here we consider the action of the operator $\mathcal{P}$ on $\mathbb{C}[A_n]$ by multiplication on the right. To obtain the eigenvalues of $\mathcal{P}$, we will use the representation theory of $A_n$. We now define the terminology that we will need. Most of the notation here is borrowed from Ruff \cite{Ruff}.

Let $V$ be an irreducible representation of $S_n$. The restriction of $V$ to $S_{n-1}$ has a multiplicity-free decomposition into irreducibles 
\cite[Theorem 2.8.3]{Sagan}. For example taking $n=15$ and $V=S^{(5,4,4,2)}$, the restriction of $V$ to $S_{14}$ is given by the following:
\[S^{(5,4,4,2)}\downarrow_{S_{14}}^{S_{15}}=S^{(4,4,4,2)}\oplus S^{(5,4,3,2)}\oplus S^{(5,4,4,1)}.\]
Again, restriction of each of these irreducibles to $S_{n-2}$ has a multiplicity-free decomposition into irreducible representations of $S_{n-2}$. Iterating this, we get a canonical decomposition of $V$ into irreducible representations of $S_1$ i.e., one-dimensional subspaces \cite[Theorem 2.9]{VO}. Thus there is a canonical basis of $V$. This basis is named the \emph{Gelfand-Tsetlin} basis (or GZ-basis) of $V$. 
Any vector from this Gelfand-Tsetlin basis is uniquely (up to scalar factor) determined by the eigenvalues of the \emph{Jucys-Murphy elements for $S_n$} on this vector.
%We also have the following:
%\begin{align*}
%S^{(4,4,4,2)}\downarrow_{S_{13}}^{S_{14}} & = S^{(4,4,3,2)}\oplus S^{(4,4,4,1)},\\
%S^{(5,4,3,2)}\downarrow_{S_{13}}^{S_{14}} & = S^{(4,4,3,2)}\oplus S^{(5,3,3,2)} \oplus S^{(5,4,2,2)}\oplus %S^{(5,4,3,1)},\\
%S^{(5,4,4,1)}\downarrow_{S_{13}}^{S_{14}} & = S^{(4,4,4,1)}\oplus S^{(5,4,3,1)}\oplus S^{(5,4,4)},
%\end{align*}
%and so on. 
The Jucys-Murphy elements $X_i=\sum_{1\leq j<i}(j,i)\in\mathbb{C}[S_n],\;1\leq i\leq n$ for $S_n$ act semisimply on the Gelfand-Tsetlin basis of the irreducible representations of $S_n$ \cite{VO}. We now define the Jucys-Murphy elements for $A_n$ and establish its connection to $\mathcal{P}$.
\begin{defn}\cite{Ruff}\label{df: Jucys-Murphy}
 The \emph{Jucys-Murphy elements} $Y_1,\dots,Y_n\in\mathbb{C}[A_n]$ are defined by $Y_1=0,\;Y_2=\1\;\text{ and }Y_i=(1,2)X_i\text{ for }i\geq 3$, where $X_i\text{ for } 1\leq i\leq n$ are the usual Jucys-Murphy elements for $S_n$.
\end{defn}
 Let us denote $s_i=(i,i+1)$ for $1\leq i<n$. Then $\{s_1,\dots,s_{n-1}\}$ is a set of generators of the symmetric group. $A_n$ is generated by $t_2,\dots,t_{n-1}$, where $t_i=(1,2)s_i$ for $i\in\{2,\dots,n-1\}$. As the generators $s_1,\dots,s_{n-1}$ of the symmetric group satisfy \[s_iX_j=X_js_i,\;s_iX_i=X_{i+1}s_i-\1\text{ for all }1\leq i<n\text{ with }|i-j|>1,\] we have the following:
\[t_iY_i=Y_{i+1}t_i-\1\;\text{ for all }3\leq i<n.\]
\begin{lem}\label{lem: relation of P and JM-elements}
$\mathcal{P}=\1$ if $n=2,\;\mathcal{P}=\1+Y_3$ if $n=3$ and for $n>3$, we have
\[\mathcal{P}=t_{n-1}\left(Y_n+Y_{n-1}\right).\]
\end{lem}
\begin{proof}
The cases of $n=2,3$ are just verification. We prove this lemma for $n>3$.
\begin{align*}
\mathcal{P}&=\1+\sum\limits_{i=1}^{n-2}\left((i,n-1,n)+(i,n,n-1)\right)\\
&=\1+\sum\limits_{i=1}^{n-2}\left((n,n-1)(n,i)+(n-1,n)(n-1,i)\right)\\
&=(n,n-1)\left((n,n-1)+\sum\limits_{i=1}^{n-2}(n,i)+\sum\limits_{i=1}^{n-2}(n-1,i)\right)\\
&=(n,n-1)\left(\sum\limits_{i=1}^{n-1}(n,i)+\sum\limits_{i=1}^{n-2}(n-1,i)\right)\\
&=(n,n-1)\left(X_n+X_{n-1}\right)\\
&=(1,2)(n,n-1)(1,2)\left(X_n+X_{n-1}\right)\\
&=t_{n-1}\left(Y_n+Y_{n-1}\right).
\end{align*}
\end{proof}
\begin{rem}
We note that $Y_iY_j=Y_jY_i$ for all $1\leq i,j\leq n$ and the common eigenvectors for $Y_i$s form a basis for the irreducible representations of $A_n$.
\end{rem}
A \emph{partition} $\lambda$ of a positive integer $n$ is a weakly decreasing finite sequence $(\lambda_1,\cdots,\lambda_r)$ of positive integers such that $\sum_{i=1}^{r}\lambda_i=n$. We write $\lambda\vdash n$ to mean $\lambda$ is a partition of $n$. A partition $\lambda$ can be pictorially visualised as a left-justified arrangement of $r$ rows of boxes with $\lambda_i$ boxes in the $i^{\text{th}}$ row. This pictorial arrangement of boxes is known as the \emph{Young diagram} of $\lambda$. For example there are five partitions of the positive integer $4$ viz. (4), (3,1), (2,2), (2,1,1) and (1,1,1,1). Young diagrams corresponding to the partitions of $4$ are the following:
\[\begin{array}{cclll}
\yng(4)&\hspace{0.5cm}\yng(3,1)& \hspace{0.5cm}\yng(2,2) & \hspace{0.5cm}\yng(2,1,1) & \hspace{0.75cm}\yng(1,1,1,1)\\
(4)\;&\;\quad(3,1)&\quad\;(2,2)&\quad(2,1,1)&\;(1,1,1,1)
\end{array}\]
A \emph{Young tableau} of shape $\lambda$ (or \emph{$\lambda$-tableau}) is obtained by filling the numbers $1,\dots,n$ in the boxes of the Young diagram of $\lambda$. A $\lambda$-tableau is \emph{standard} if the entries in its boxes increase from left to right along rows and from top to bottom along columns. The set of all standard tableaux of a given shape $\lambda$ is denoted by $\std$. For example, standard Young tableaux of shape $(3,1)$ are:
\[\begin{array}{ccc}
\young({{\substack{1}}}{{\substack{2}}}{{\substack{3}}},{{\substack{4}}}) &\quad\young({{\substack{1}}}{{\substack{2}}}{{\substack{4}}},{{\substack{3}}}) & \quad \young({{\substack{1}}}{{\substack{3}}}{{\substack{4}}},{{\substack{2}}})
\end{array}\]
The \emph{content} of a box in row $i$ and column $j$ of a diagram is the integer $j-i$. Given a tableau $T$, $c(T,i)$ denotes the content of the box labelled $i$ in $T$, $1\leq i\leq n$. Given a Young diagram $\lambda$, its \emph{conjugate} $\lambda^{\prime}$ is obtained by reflecting $\lambda$ with respect to the diagonal consisting of boxes with content $0$. A diagram $\lambda$ is \emph{self-conjugate} if $\lambda^{\prime}=\lambda$. An \emph{upper standard Young tableau} $T$ is a standard Young tableau $T$ such that $c(T,2)=1$. The collection of all upper standard tableaux of a given shape $\lambda$ is denoted by $\Ustd$.

\begin{lem}
 For $n>1$, the cardinality of $\cup_{\lambda\vdash n}\Ustd$ is half the cardinality of $\cup_{\lambda\vdash n}\std$. Moreover, for self-conjugate $\lambda\vdash n$ we have,  $|\Ustd|=\frac{1}{2}|\std|.$
\end{lem}
\begin{proof}
Let us consider the set $\Lstd:=\{T\in\std\mid c(T,2)=-1\}$. Then by sending each element of $\cup_{\lambda\vdash n}\Ustd$ to its transpose (i.e. reflecting it with respect to the diagonal containing boxes with content $0$), we have a one to one correspondence between $\cup_{\lambda\vdash n}\Ustd$ and $\cup_{\lambda\vdash n}\Lstd$. Also, it is obvious that $\left(\cup_{\lambda\vdash n}\Ustd\right)\cap\left(\cup_{\lambda\vdash n}\Lstd\right)=\phi$ and $\left(\cup_{\lambda\vdash n}\Ustd\right)\cup\left(\cup_{\lambda\vdash n}\Lstd\right)=\;\cup_{\lambda\vdash n}\std$. Hence, $|\cup_{\lambda\vdash n}\Ustd|=\frac{1}{2}|\cup_{\lambda\vdash n}\std|$. 

Also, for self-conjugate $\lambda\vdash n,\;\Ustd\cup\Lstd=\std$ and the same map as above gives a bijection from $\Lstd$ to $\Ustd$). Therefore,  $|\Ustd|=\frac{1}{2}|\std|$.
\end{proof}
We now describe all the irreducible representations of $A_n$ (for more details, see \cite{Ruff}). Corresponding to each non-self-conjugate partition $\lambda$ of $n$, there is an irreducible representation $D_{\lambda}$ of $A_n$. Given any non-self-conjugate partition $\lambda$ of $n$, the irreducible representations $D_{\lambda}$ and $D_{\lambda^{\prime}}$ of $A_n$ are isomorphic. For each self-conjugate partition $\lambda$ of $n$, there are two non-isomorphic irreducible representations $D^+_{\lambda}$ and $D^-_{\lambda}$ of $A_n$. All the irreducible representations of $A_n$ are given by $D_{\lambda}$, $\lambda\vdash n$ non-self-conjugate and $D^{\pm}_{\lambda}$, $\lambda\vdash n$ self-conjugate. The basis of $D_{\lambda}$ is identified by the elements of $\Ustd\cup\Ustdc$ for non-self-conjugate $\lambda\vdash n$ and that of $D^{\pm}_{\lambda}$ are identified by the elements of $\Ustd$ for self-conjugate $\lambda\vdash n$. We denote the dimension of $D_{\lambda}$ by $d_{\lambda}$ (when $\lambda\vdash n$ is non-self-conjugate) and dimensions of $D^{\pm}_{\lambda}$ by $d^{\pm}_{\lambda}$ (for self-conjugate $\lambda\vdash n$) respectively. Therefore, we have the following:
\begin{align*}
 d_{\lambda}= & |\Ustd|+|\Ustdc|=|\std|\;\text{ for non-self-conjugate }\lambda\vdash n
 \text{ and }\\
& d_{\lambda}^+ =d_{\lambda}^-=|\Ustd|=\frac{1}{2}|\std|\; \text{ for self-conjugate }\lambda\vdash n.
\end{align*}
Let $\Par$ be the set of all partitions of $n$. Let us consider two subsets $\Cpar$ and $\Ncpar$ defined as follows:
\begin{align*}
\Cpar&=\{\lambda\in\Par\mid\lambda=\lambda^{\prime}\}\;\text{ and }\\
\Ncpar&=\bigg\{\lambda\in\Par\bigg|\lambda\neq\lambda^{\prime},
\begin{minipage}{2.5in}
$\lambda$ contains more boxes of positive content than that of $\lambda^{\prime}$
\end{minipage}\bigg\}.
\end{align*}
In the regular representation of a finite group $G$, each irreducible representation of $G$ occurs with multiplicity equal to its dimension \cite[section 2.4]{Serre}. Therefore, from the above discussion, we have the following:
\begin{equation}\label{eq:A_n module decomposition}
\mathbb{C}[A_n]\cong\left(\underset{\lambda\in\Ncpar}{\oplus}d_{\lambda}D_{\lambda}\right)\oplus\left(\underset{\lambda\in\Cpar}{\oplus}d^+_{\lambda}D^+_{\lambda}\right)\oplus\left(\underset{\lambda\in\Cpar}{\oplus}d^-_{\lambda}D^-_{\lambda}\right).
\end{equation}

Now we discuss the actions of the generators $t_i,\;3\leq i\leq n-1$ and the Jucys-Murphy elements on the irreducible representations of $A_n$ (we don't need the action of $t_2$ on irreducible representations of $A_n$ in this work). Given any partition $\lambda$ of $n$, let us define $\alpha=(a_1,\dots,a_n):=(c(T_{\alpha},1),\dots,c(T_{\alpha},n))$, where  $T_{\alpha}\in \Ustd\cup\Ustdc$ ($=\Ustd$, if $\lambda$ is self-conjugate). Ruff in \cite{Ruff} showed that for $\lambda\in\Ncpar$, if $v_{\alpha}$ is the basis element of $D_{\lambda}$ corresponding to $T_{\alpha}\in \Ustd\cup\Ustdc$, then $Y_iv_{\alpha}=a_iv_{\alpha}$ for all $1\leq i\leq n$. Moreover for $3\leq i<n$, the action of $t_i$ on $v_{\alpha}$ is given as follows:
\begin{equation}\label{t_i action on basis}
t_iv_{\alpha}=\frac{1}{a_{i+1}-a_i}\;v_{\alpha}+\sqrt{-1}\;(-1)^{\alpha,i}\sqrt{1-\frac{1}{(a_{i+1}-a_i)^2}}\;\;v_{t_i\alpha},
\end{equation}
where 
\begin{equation*}
(-1)^{\alpha,i}=
\begin{cases}
1\quad\text{if }a_i<a_{i+1},\\
-1\quad\text{if }a_i>a_{i+1},
\end{cases}\quad\text{ and }
\end{equation*}
$t_i\alpha:=(a_1,\dots,a_{i-1},a_{i+1},a_{i},a_{i+2},\dots a_n)$ if $a_{i+1}\neq a_i\pm 1$. We don't need $t_i\alpha$ when $a_{i+1}= a_i\pm 1$, because the coefficient of $v_{t_i\alpha}$ in the expression \eqref{t_i action on basis} is zero in that case.

Also for $\lambda\in\Cpar$, if $v^+_{\alpha},\;v^-_{\alpha}$ are the basis elements of $D^+_{\lambda},\;D^-_{\lambda}$ respectively, corresponding to $T_{\alpha}\in \Ustd$, then the actions of $Y_i,\;1\leq i\leq n$ and $t_i,\; 3\leq i<n$ on $v^{\pm}_{\alpha}$ are same as their respective actions on $v_{\alpha}$ in case of $\lambda\in\Ncpar$. Now we are in a position to find the eigenvalues of $\mathcal{P}$. 
\begin{thm}\label{thm:eigenvalues non-conj case}
Let $\lambda\vdash n$ be a non-self-conjugate partition of $n$. For each $T_{\alpha}\in \Ustd\cup\Ustdc$, we have an eigenvalue of $\mathcal{P}$, given as follows:
\begin{enumerate}
\item $2a_n-1$ is an eigenvalue of $\mathcal{P}$, if $a_n=a_{n-1}+1$.
\item $-(2a_n+1)$ is an eigenvalue of $\mathcal{P}$, if $a_n=a_{n-1}-1$.
\item $\pm(a_n+a_{n-1})$ are eigenvalues of $\mathcal{P}$, if $a_n\neq a_{n-1}\pm1\text{ and } a_{n-1}<a_n$.
\end{enumerate}
Here, we have used $\alpha=(a_1,\dots,a_n):=(c(T_{\alpha},1),\dots,c(T_{\alpha},n))$. Moreover, each eigenvalue has multiplicity $d_{\lambda}$. Also, note that the sets of eigenvalues obtained by considering the partitions $\lambda$ and $\lambda^{\prime}$ are same.
\end{thm}
\begin{proof}
The theorem is trivially true for $n=2,3$. Now we prove the theorem for $n>3$. For any non-self-conjugate $\lambda\vdash n$, we can choose basis element $v_{\alpha}$ of $D_{\lambda}$ corresponding to $T_{\alpha}\in\Ustd\cup\Ustdc$ such that $Y_iv_{\alpha}=a_iv_{\alpha}\text{ for all }i=1,\dots,n$. Now a basis $\mathcal{B}$ of $D_{\lambda}$ is the union of the following three sets:
\[\mathcal{B}_1:=\{v_{\alpha}\mid a_n=a_{n-1}+1\},\;\mathcal{B}_2:=\{v_{\alpha}\mid a_n=a_{n-1}-1\},\;\mathcal{B}_3:=\{v_{\alpha}\mid a_n\neq a_{n-1}\pm1\}.\]
For any upper standard Young tableau $T_{\alpha}\in\mathcal{B}_3$, we have another upper standard Young tableau $T_{t_{n-1}\alpha}\in\mathcal{B}_3$. Therefore, the cardinality of $\mathcal{B}_3$ is even and  $\mathcal{B}_3=\{v_{\alpha},\;v_{t_{n-1}\alpha}\mid a_n\neq a_{n-1}\pm1,\;a_{n-1}<a_n\}.$ Again for $a_n\neq a_{n-1}\pm1$, from \eqref{t_i action on basis}, we have 
\[\mathbb{C}\text{-Span }\{v_{\alpha},\;v_{t_{n-1}\alpha}\}=\mathbb{C}\text{-Span }\{v_{\alpha},\;t_{n-1}v_{\alpha}\}.\]
Therefore $\mathcal{B}_1\cup\mathcal{B}_2\cup\{v_{\alpha},\;t_{n-1}v_{\alpha}\mid a_n\neq a_{n-1}\pm1,\;a_{n-1}<a_n\}$ is a basis for $D_{\lambda}$. Let us consider the ordered basis $\mathcal{B}^{'}$ of $D_{\lambda}$ in which we first collect all the vectors from $\mathcal{B}_1$. Then all the vectors from $\mathcal{B}_2$ and finally the pair of vectors $(v_{\alpha},\;t_{n-1}v_{\alpha})$ from $\{v_{\alpha},\;t_{n-1}v_{\alpha}\mid a_n\neq a_{n-1}\pm1,\;a_{n-1}<a_n\}$. For $v_{\alpha}\in\mathcal{B}_1$,
\begin{align*}
\mathcal{P}v_{\alpha}&=t_{n-1}\left(Y_{n-1}+Y_n\right)v_{\alpha},\quad\text{ by Lemma \ref{lem: relation of P and JM-elements}}\\
&=(a_{n-1}+a_n)t_{n-1}v_{\alpha}\\
&=(a_{n-1}+a_n)v_{\alpha},\quad\text{ by \eqref{t_i action on basis}}\\
&=(2a_n-1)v_{\alpha}.
\end{align*}
Again for $v_{\alpha}\in\mathcal{B}_2$,
\begin{align*}
\mathcal{P}v_{\alpha}&=t_{n-1}\left(Y_{n-1}+Y_n\right)v_{\alpha},\quad\text{ by Lemma \ref{lem: relation of P and JM-elements}}\\
&=(a_{n-1}+a_n)t_{n-1}v_{\alpha}\\
&=-(a_{n-1}+a_n)v_{\alpha},\quad\text{ by \eqref{t_i action on basis}}\\
&=-(2a_n+1)v_{\alpha}.
\end{align*}
Therefore $\mathcal{P}$ acts on $\mathcal{B}_1$ and $\mathcal{B}_2$ as scalars. Now for $v_{\alpha}\in\mathcal{B}_3$, using $t_{n-1}Y_{n-1}=Y_nt_{n-1}-\1$ and $t_{n-1}^2=\1$, the matrix for the action of $\mathcal{P}$ on $\{v_{\alpha},\;t_{n-1}v_{\alpha}\}$ is given below,
\begin{equation}\label{eq: 2X2 blocks}
\left(
{\begin{array}{cc}
0 & a_{n-1}+a_n\\
a_{n-1}+a_n & 0 \end{array}}
\right).
\end{equation}
The eigenvalues of the above $2\times 2$ matrix are $\pm(a_{n-1}+a_n)$. Therefore, the matrix of $\mathcal{P}$ with respect to the ordered basis $\mathcal{B}^{'}$ is a block diagonal matrix, where first $|\mathcal{B}_1|$ blocks are the $1\times 1$ matrix $(2a_n-1)$, next $|\mathcal{B}_2|$ blocks are the $1\times 1$ matrix $-(2a_n+1)$ and last $|\{\alpha: a_n\neq a_{n-1}\pm1,\;a_{n-1}<a_n\}|$ blocks are the $2\times 2$ matrix \eqref{eq: 2X2 blocks}. The argument for the multiplicity of the eigenvalues follows from \eqref{eq:A_n module decomposition} .Thus the theorem follows.
\end{proof}
\begin{thm}\label{thm:eigenvalues conj case}
For a self-conjugate partition $\lambda\vdash n$, each $T_{\alpha}\in \Ustd$ provides an eigenvalue of $\mathcal{P}$, given as follows:
\begin{enumerate}
\item $2a_n-1$ is an eigenvalue of $\mathcal{P}$, if $a_n=a_{n-1}+1$.
\item $-(2a_n+1)$ is an eigenvalue of $\mathcal{P}$, if $a_n=a_{n-1}-1$.
\item $\pm(a_n+a_{n-1})$ are eigenvalues of $\mathcal{P}$, if $a_n\neq a_{n-1}\pm1\text{ and } a_{n-1}<a_n$.
\end{enumerate}
Here, we have used $\alpha=(a_1,\dots,a_n):=(c(T_{\alpha},1),\dots,c(T_{\alpha},n))$. Moreover, each eigenvalue has multiplicity $|\std|=d^+_{\lambda}+d^-_{\lambda}$.
\end{thm}
\begin{proof}
The proof is similar to the proof of Theorem \ref{thm:eigenvalues non-conj case}. Proof of this theorem follows by replacing $D_{\lambda}$, $v_{\alpha}$ and $d_{\lambda}$ by $D^{\pm}_{\lambda}$, $v^{\pm}_{\alpha}$ and $d^{\pm}_{\lambda}$ respectively, in the proof of Theorem \ref{thm:eigenvalues non-conj case}.
\end{proof}
\begin{cor}\label{thm:eigenvalues combined case}
Let $\lambda$ be a partition of $n$. Then each $T_{\alpha}\in \Ustd\cup\Ustdc$ provides an eigenvalue of $\mathcal{P}$, given as follows:
\begin{enumerate}
	\item $2a_n-1$ is an eigenvalue of $\mathcal{P}$, if $a_n=a_{n-1}+1$.
	\item $-(2a_n+1)$ is an eigenvalue of $\mathcal{P}$, if $a_n=a_{n-1}-1$.
	\item $\pm(a_n+a_{n-1})$ are eigenvalues of $\mathcal{P}$, if $a_n\neq a_{n-1}\pm1\text{ and } a_{n-1}<a_n$.
\end{enumerate}
Here, we have used $\alpha=(a_1,\dots,a_n):=(c(T_{\alpha},1),\dots,c(T_{\alpha},n))$. Moreover, each eigenvalue has multiplicity $d_{\lambda}:=|\std|$.
\end{cor}
\begin{proof}
If $\lambda\vdash n$ is non-self-conjugate, then the result follows directly from Theorem \ref{thm:eigenvalues non-conj case}. Now if $\lambda\vdash n$ is self-conjugate, then we have $\Ustd=\Ustd\cup\Ustdc$. Therefore, in the case of self-conjugate $\lambda$, this result follows from Theorem \ref{thm:eigenvalues conj case}.
\end{proof}
\begin{rem}
Corollary \ref{thm:eigenvalues combined case}, together with \eqref{eq:A_n module decomposition}, determines the spectrum of $\mathcal{P}$ and hence that of $\widehat{P}(R)$.
\end{rem}
\begin{exmp}
If $n=4$, then eigenvalues of $\widehat{P}(R)$  are the following: 
\[\left(1,\;\frac{3}{5},\;\frac{3}{5},\;\frac{3}{5},\;\frac{1}{5},\;\frac{1}{5},\;\frac{1}{5},\;-\frac{1}{5},\;-\frac{1}{5},\;-\frac{1}{5},\;-\frac{1}{5},\;-\frac{1}{5}\right).\]
\end{exmp}
For the only element $T_{(0,1,2,3)}=\begin{array}{c}\young({{\substack{1}}}{{\substack{2}}}{{\substack{3}}}{{\substack{4}}})\end{array}$ of   $\Ustdexample((4))\cup\Ustdexample((1,1,1,1))$ we have, $a_3=2,\;a_4=3$. Hence $a_4=a_3+1$ and the eigenvalue of $\widehat{P}(R)$ corresponding to $T_{(0,1,2,3)}$ is $1$ with multiplicity $1$. The elements of $\Ustdexample((3,1))\cup\Ustdexample((2,1,1))$ are listed below:
\[T_{(0,1,2,-1)}=\begin{array}{c}\young({{\substack{1}}}{{\substack{2}}}{{\substack{3}}},{{\substack{4}}})\end{array},\;T_{(0,1,-1,2)}=\begin{array}{c}\young({{\substack{1}}}{{\substack{2}}}{{\substack{4}}},{{\substack{3}}})\end{array},\;T_{(0,1,-1,-2)}=\begin{array}{c}\young({{\substack{1}}}{{\substack{2}}},{{\substack{3}}},{{\substack{4}}})\end{array}.\]
Now $a_3=2,\;a_4=-1$ for $T_{(0,1,2,-1)}$ and $a_3=-1,\;a_4=2$ for $T_{(0,1,-1,2)}$. Thus for both $T_{(0,1,2,-1)}$ and $T_{(0,1,-1,2)}$ we have $a_4\neq a_3\pm 1$. In order to satisfying $a_3<a_4$ we choose $T_{(0,1,-1,2)}$ and the eigenvalues of $\widehat{P}(R)$ in this case are $\pm\frac{1}{5}$ with multiplicity $3$ each. Again for $T_{(0,1,-1,-2)}$ we have $a_3=-1,\;a_4=-2$ which satisfies $a_4=a_3-1$. Thus the eigenvalue of $\widehat{P}(R)$ corresponding to $T_{(0,1,-1,-2)}$ is $\frac{3}{5}$ with multiplicity $3$. Finally considering the only element $T_{(0,1,-1,0)}=\begin{array}{c}\young({{\substack{1}}}{{\substack{2}}},{{\substack{3}}}{{\substack{4}}})\end{array}$ of $\Ustdexample((2,2))$ we have $a_3=-1,\;a_4=0$ and thus $a_4=a_3+1$. Therefore the eigenvalue of $\widehat{P}(R)$ corresponding to $T_{(0,1,-1,0)}$ is $-\frac{1}{5}$ with multiplicity $2$.
\begin{exmp}
Given $n\geq 4$, the eigenvalues of $\widehat{P}(R)$ for the $(n-1)$-dimensional irreducible representation $D_{(n-1,1)}\;(\text{or }D_{(2,1^{n-2})})$ are given as follows:
\begin{center}
		\begin{tabular}{cccc}
		\emph{Eigenvalues:} & $\frac{n-3}{2n-3}$ &\quad $-\frac{n-3}{2n-3}$ &\quad $\frac{2n-5}{2n-3}$\\
    	\emph{Multiplicities:}& $n-1$ &\quad $n-1$ &\quad $(n-3)(n-1)$
		\end{tabular}
\end{center}
\end{exmp}
First consider the elements 
\[\young({{\substack{1}}}{{\substack{2}}}{{\substack{\cdots}}}{{\substack{n\text{-}1}}}{{\substack{n}}},{{\substack{i}}}),\quad3\leq i \leq n-2\]
of $\Ustdexample((n-1,1))$. Each of these elements satisfies $a_n=a_{n-1}+1$ as $a_{n-1}=n-3,\;a_n=n-2$.
Therefore the eigenvalues of $\widehat{P}(R)$ corresponding to each of these elements are $\frac{2n-5}{2n-3}$ with multiplicity $n-1$ each. There are $n-4$ such upper standard tableaux, thus multiplicity of the eigenvalue $\frac{2n-5}{2n-3}$ is $(n-4)(n-1)$. Now considering the element 
$\begin{array}{c}\young({{\substack{1}}}{{\substack{2}}}{{\substack{\cdots}}}{{\substack{n\text{-}1}}},{{\substack{n}}})\end{array}$ of $\Ustdexample((n-1,1))$, we have $a_{n-1}=n-2,\;a_n=-1$ and thus $a_n\neq a_{n-1}\pm 1$ but $a_{n-1}\nless a_n$. Therefore we will not select this upper standard tableaux. For the element
$\begin{array}{c}
\young({{\substack{1}}}{{\substack{2}}}{{\substack{\cdots}}}{{\substack{n}}},{{\substack{n\text{-}1}}})
\end{array}$ of $\Ustdexample((n-1,1))$, we have $a_{n-1}=-1,\;a_n=n-2$. Hence this upper standard tableaux satisfies $a_n\neq a_{n-1}\pm 1$ and $a_{n-1}< a_n$. Therefore the eigenvalues of $\widehat{P}(R)$ corresponding to this upper standard tableaux are $\pm\frac{n-3}{2n-3}$ with multiplicity $n-1$ each. Finally we consider the only element
\[\young({{\substack{1}}}{{\substack{2}}},{{\substack{3}}},{{\substack{4}}},{{\smash{\vdots}}},{{\substack{n\text{-}1}}},{{\substack{n}}})\] of $\Ustdexample(2,1^{n-2})$. This upper standard tableaux satisfies $a_n=a_{n-1}-1$ as $a_{n-1}=-(n-3),\;a_n=-(n-2)$. Thus the eigenvalue of $\widehat{P}(R)$ corresponding to this upper standard tableaux is $\frac{2n-5}{2n-3}$ with multiplicity $n-1$. %------end of blue colouring--- %Therefore combining all these cases we have the illustration. 

%----------------UPPER BOUNDS---------
\section{Upper Bound for total variation distance}\label{sec:upper bound}
In this section, we prove Theorem \ref{thm:UB for Alt}. We first state the \emph{Upper bound lemma}.
\begin{lem}[Upper bound lemma]\cite[p. 16, lemma 4.2]{Diaconis Note1}\label{Upper Bound Lemma}
Let $p$ be a probability measure on a finite group  $G$ such that $p(x)=p(x^{-1})$ for all $x\in G$. % and $U_{G}$ be the uniform distribution on $G$. 
Suppose the random walk on $G$ driven by $p$ is irreducible. Then we have the following
\[||p^{*k}-U_{G}||^2_{\text{TV}}\leq\frac{1}{4}\sum\limits_{\rho}^{*}d_{\rho}\left(\emph{Trace}\left(\left(\widehat{p}(\rho)\right)^{2k}\right)\right),\]
where the sum is over all non-trivial irreducible representations  $D_{\rho}$ of $G$ and $d_{\rho}$ is the dimension of $D_{\rho}$. % and $A_{\rho}$ is the matrix $\widehat{p}(\rho)\;($i.e., matrix of $\widehat{p}(\rho)$ acting on $D_{\rho})$.
\end{lem}
\begin{proof}[Proof of Theorem \ref{thm:UB for Alt}]
We know that the trace of the $2k^{\text{th}}$ power of a matrix is the sum of $2k^{\text{th}}$ powers of its eigenvalues. Now for $\lambda\vdash n$ we can say by Corollary \ref{thm:eigenvalues combined case}, as $T_{\alpha}$ ranges over $\Ustd\cup\Ustdc$, the eigenvalues of $\widehat{P}(R)$ are
\begin{itemize}
\item $\frac{2a_n-1}{2n-3}=\frac{a_n+a_{n-1}}{2n-3}$, if $a_n=a_{n-1}+1$,
\item $-\frac{(2a_n+1)}{2n-3}=-\frac{a_n+a_{n-1}}{2n-3}$, if $a_n=a_{n-1}-1$,
\item $\pm\frac{(a_n+a_{n-1})}{2n-3}=\pm\frac{a_n+a_{n-1}}{2n-3}$ if $a_n\neq a_{n-1}\pm1\text{ and } a_{n-1}<a_n$,
\end{itemize}
where $a_{n-1}$ and $a_n$ are contents of $(n-1)$ and $n$ respectively, in $T_{\alpha}$. Lemma \ref{Upper Bound Lemma} implies:
\begin{align}\label{eq:UB1}
\begin{split}
\;4\;&||P^{*k}-U_{A_n}||^2_{\text{TV}}\leq\sum\limits_{\lambda\in \Ncpar\setminus\{(n)\}}\left(d_{\lambda}\sum\limits_{T_{\alpha}\in\Ustd\cup\Ustdc}\left(\frac{a_n+a_{n-1}}{2n-3}\right)^{2k}\right)\\
&\;\;+\sum\limits_{\lambda\in\Cpar}\left(d^+_{\lambda}\sum\limits_{T_{\alpha}\in\Ustd}\left(\frac{a_n+a_{n-1}}{2n-3}\right)^{2k}+d^-_{\lambda}\sum\limits_{T_{\alpha}\in\Ustd}\left(\frac{a_n+a_{n-1}}{2n-3}\right)^{2k}\right).
\end{split}
\end{align}
Before coming to the main part of the proof, let us consider the leading term in \eqref{eq:UB1}, which corresponds to the partition $\lambda=(n-1,1)$ or equivalently its conjugate. For the partition $\lambda=(n-1,1)$, the eigenvalues are $\frac{2n-5}{2n-3},\;-\frac{n-3}{2n-3}\text{ and }\frac{n-3}{2n-3}$ with algebraic multiplicities $n-3,\;1$ and $1$, respectively. Therefore, the term in $\sum\limits_{\lambda\in \Ncpar\setminus\{(n)\}}$ is
\begin{align*}
&\quad\;\;(n-1)\left((n-3)\left(\frac{2n-5}{2n-3}\right)^{2k}+\left(-\frac{n-3}{2n-3}\right)^{2k}+\left(\frac{n-3}{2n-3}\right)^{2k}\right)\\
&=(n-1)\left((n-3)\left(1-\frac{2}{2n-3}\right)^{2k}+2\left(\frac{n-3}{2n-3}\right)^{2k}\right)=O\left(e^{-\frac{4k}{2n-3}+2\log n}\right).
\end{align*}
Now, if $k=(n-\frac{3}{2})(\log n+c)$, then $e^{-\frac{4k}{2n-3}+2\log n}$ is $e^{-2c}$, $c>0$. We will show that this is the largest term, other terms being smaller.

For any upper standard Young tableau $T_{\alpha}$ of shape $\lambda$, if $\lambda=(\lambda_1,\dots,\lambda_r)$, then $a_{n-1}+a_n\leq 2\lambda_1-3$. Let $\mathcal{E}=e^{-\frac{1}{2}}(n-\frac{3}{2})+\frac{3}{2}$. Then the right hand side of \eqref{eq:UB1} is less than or equal to
\begin{align*}
&\sum_{\lambda\in \Ncpar\setminus\{(n)\}}\left(d_{\lambda}\sum_{T_{\alpha}\in\Ustd\cup\Ustdc}\left(\frac{2\lambda_1-3}{2n-3}\right)^{2k}\right)\\
&\quad\quad+\sum_{\lambda\in\Cpar}\left(d^+_{\lambda}\sum_{T_{\alpha}\in\Ustd}\left(\frac{ 2\lambda_1-3}{2n-3}\right)^{2k}+d^-_{\lambda}\sum_{T_{\alpha}\in\Ustd}\left(\frac{ 2\lambda_1-3}{2n-3}\right)^{2k}\right).
\end{align*}
The summands in $\sum\limits_{T_{\alpha}\in\Ustd\cup\Ustdc}$ and $\sum\limits_{T_{\alpha}\in\Ustd}$ are independent of the index of the sum. Therefore, the above expression becomes
\begin{align}\label{eq:UB1.0}
&\sum_{\lambda\in \Ncpar\setminus\{(n)\}}\left(d_{\lambda}^2\left(\frac{2\lambda_1-3}{2n-3}\right)^{2k}\right)+\sum_{\lambda\in\Cpar}\left(\left(d^+_{\lambda}+d^-_{\lambda}\right)\frac{d_{\lambda}}{2}\left(\frac{ 2\lambda_1-3}{2n-3}\right)^{2k}\right)\nonumber\\
&=\sum_{\lambda\in \Ncpar\setminus\{(n)\}}d_{\lambda}^2\left(\frac{2\lambda_1-3}{2n-3}\right)^{2k}+\sum_{\lambda\in\Cpar}\frac{d^2_{\lambda}}{2}\left(\frac{ 2\lambda_1-3}{2n-3}\right)^{2k}.
\end{align} 
Now adding the non-negative quantities $\displaystyle\sum_{\substack{\lambda\vdash n\\\lambda\notin\Ncpar\cup\Cpar}}d_{\lambda}^2\left(\frac{2\lambda_1-3}{2n-3}\right)^{2k}$ and $\sum\limits_{\lambda\in\Cpar}\frac{d^2_{\lambda}}{2}\left(\frac{ 2\lambda_1-3}{2n-3}\right)^{2k}$ to \eqref{eq:UB1.0}, we obtain $\displaystyle\sum_{\substack{\lambda\vdash n\\\lambda\neq(n)}}d_{\lambda}^2\left(\frac{2\lambda_1-3}{2n-3}\right)^{2k}$. Therefore, the expression in \eqref{eq:UB1.0} is less than or equal to
\begin{align}\label{eq: from this onward almost same pf}
\sum_{\substack{\lambda\vdash n\\\lambda\neq(n)}}d_{\lambda}^2\left(\frac{\lambda_1-\frac{3}{2}}{n-\frac{3}{2}}\right)^{2k}\nonumber&= \sum_{\substack{\lambda\vdash n\\\lambda\neq(n)\\\lambda_1>\mathcal{E}}}d_{\lambda}^2\left(\frac{\lambda_1-\frac{3}{2}}{n-\frac{3}{2}}\right)^{2k}+ \sum_{\substack{\lambda\vdash n\\\lambda_1\leq \mathcal{E}}}d_{\lambda}^2\left(\frac{\lambda_1-\frac{3}{2}}{n-\frac{3}{2}}\right)^{2k}\nonumber\\
&\leq \sum_{\substack{\lambda\vdash n\\\lambda_1=t}}^{n-1}d_{\lambda}^2\left(\frac{\lambda_1-\frac{3}{2}}{n-\frac{3}{2}}\right)^{2k}+\sum_{\substack{\lambda\vdash n\\\lambda_1\leq \mathcal{E}}}d_{\lambda}^2e^{-k}.
\end{align}
Here, $t$ is the smallest integer greater than or equal to $\mathcal{E}$. Given $\lambda=(\lambda_1,\dots,\lambda_r)\vdash n$, we have $d_{\lambda}\leq {n\choose\lambda_1}d_{\mu}$, for $\mu\vdash(n-\lambda_1)$  such that largest part of $\mu\leq\lambda_1$. Thus expression \eqref{eq: from this onward almost same pf} is less than or equal to
\begin{align}\label{eq: UB1.1}
&\quad\sum\limits_{\lambda_1=t}^{n-1} \sum_{\substack{\mu\vdash (n-\lambda_1)\\\text{ largest part of }\mu\leq\lambda_1}}{n\choose\lambda_1}^2d_{\mu}^2\left(\frac{\lambda_1-\frac{3}{2}}{n-\frac{3}{2}}\right)^{2k}+\sum_{\substack{\lambda\vdash n\\\lambda_1\leq \mathcal{E}}}d_{\lambda}^2e^{-k}\nonumber\\
&\leq\sum\limits_{\lambda_1=t}^{n-1}{n\choose\lambda_1}^2\left(\frac{\lambda_1-\frac{3}{2}}{n-\frac{3}{2}}\right)^{2k} \sum_{\mu\vdash (n-\lambda_1)}d_{\mu}^2+\sum_{\substack{\lambda\vdash n}}d_{\lambda}^2e^{-k}\nonumber\\
&= \sum\limits_{m=1}^{n-t}{n\choose m}^2\left(1-\frac{m}{n-1.5}\right)^{2k}m! +n!e^{-k},\quad\text{ where }m=n-\lambda_1.
\end{align}
Since $1\geq\frac{m}{n-1.5}\geq 0$ and $1-x\leq e^{-x}$ for all non-negative real $x$, expression \eqref{eq: UB1.1} is less than or equal to
\begin{align*}
&\quad\sum\limits_{m=1}^{n-t}\frac{\left(n(n-1)\dots(n-m+1)\right)^2}{m!}e^{-\frac{2km}{n-1.5}}+n!e^{-k}\\
&\leq\sum\limits_{m=1}^{n-t}\frac{n^{2m}}{m!}e^{-\frac{2km}{n-1.5}}+n!e^{-k} < e^{n^2e^{-\frac{2k}{n-1.5}}}-1+n!e^{-k}.
\end{align*}
Therefore we have
\begin{equation}\label{eq:UB1.2}
4\;||P^{*k}-U_{A_n}||^2_{\text{TV}}<\;e^{n^2e^{-\frac{2k}{n-1.5}}}-1+n!e^{-k}.
\end{equation}
Now for $c>0$ and $k\geq(n-\frac{3}{2})(\log n+c)$, we have
\begin{equation*}
4\;||P^{*k}-U_{A_n}||^2_{\text{TV}}
\leq e^{e^{-2c}}-1+\frac{n!e^{-(n-1.5)c}}{n^{(n-1.5)}}
\leq 2e^{-2c}+\frac{n!e^{-(n-1.5)c}}{n^{(n-1.5)}}.
\end{equation*}
Therefore, $||P^{*k}-U_{A_n}||_{\text{TV}}<\frac{1}{\sqrt{2}}e^{-c}+\frac{1}{2}\sqrt{\frac{n!e^{-(n-1.5)c}}{n^{(n-1.5)}}}$ (since $\sqrt{a+b}<\sqrt{a}+\sqrt{b}$ for any two positive real numbers $a$ and $b$). This proves the first part. 

Now let $\epsilon\in (0,1)$ and $k_n=\lfloor(1+\epsilon)(n-\frac{3}{2})\log n\rfloor$. Then \eqref{eq:UB1.2} implies,
\begin{equation}\label{eq: UB 1.3}
0\leq4\;||P^{*k_n}-U_{A_n}||^2_{\text{TV}}<\;e^{n^2e^{-\frac{2k_n}{n-1.5}}}-1+n!e^{-k_n}
\leq e^{n^{-2\epsilon}e^{\frac{2}{n-1.5}}}-1+\frac{n!e}{n^{(1+\epsilon)(n-1.5)}},
\end{equation}
the last inequality of \eqref{eq: UB 1.3} holds because of $k_n+1\geq(1+\epsilon)(n-\frac{3}{2})\log n$. Thus the second part follows from the fact that the right hand side of  \eqref{eq: UB 1.3} tends to $0$ as $n\rightarrow\infty$.
\end{proof}
%------------------LOWER BOUND FOR TOTAL VARIATION DISTANCE--------------
\section{Lower Bound for total variation distance}\label{sec:lower bound}
In this section, we will find a lower bound of the total variation distance $||P^{*k}-U_{A_n}||_{\text{TV}}$ when $k=(n-\frac{3}{2})(\log n+c)$ for $c\ll 0$. To prove the results, we consider the slow term in the upper bound lemma \cite{Diaconis Note1}. The slow term comes from the $(n-1)$-dimensional irreducible representation of $A_n$. In particular, we will define a random variable on $A_n$ giving the number of fixed points of even permutations. This random variable can be viewed as the character of the restriction of \emph{defining} representation $\rho^{\text{def}}$ to $A_n$. The restriction of defining representation decomposes into two irreducible representations of $A_n$ namely the trivial representation and the $(n-1)$-dimensional representation. Thus the character of the $(n-1)$-dimensional irreducible representation plays an important role in this section.
%we use some tools from representation theory of finite groups, some of which have been discussed in Sections \ref{intro} and \ref{sec:representation} and the others will be discussed now.

We have seen all the irreducible representation of $A_n$ in Section \ref{sec:representation}. We also know from \cite[Theorem 2.4.6]{Sagan} that all $\lambda\vdash n$ give all non-isomorphic irreducible representations $S^{\lambda}$ of $S_n$. Also \cite[Theorem 4.6.5]{Amri} says that the restriction of the irreducible representation $S^{\lambda}$ of $S_n$ to $A_n$ is an irreducible representation of $A_n$ if $\lambda\neq\lambda^{\prime}$ and a direct sum of two non-isomorphic irreducible representations of $A_n$ if $\lambda=\lambda^{\prime}$. Let $\textbf{[n]}=\{\textbf{1,}\dots,\textbf{ n}\}$ and $\mathbb{C}\textbf{[n]}=\{c_1\textbf{1}+c_2\textbf{2}+\cdots+c_n\textbf{n}\mid c_i\in\mathbb{C}\text{ for all }i\}$. Then we define two representations $\rho^{\text{def}}: S_n\rightarrow GL(\mathbb{C}\textbf{[n]})$ and $\rho^{\text{def}}\otimes\rho^{\text{def}}: S_n\rightarrow GL(\mathbb{C}\textbf{[n]}\otimes\mathbb{C}\textbf{[n]})$ of $S_n$ as follows:
\begin{align*}
&\rho^{\text{def}}(\pi)\left(c_1\textbf{1}+c_2\textbf{2}+\cdots+c_n\textbf{n}\right)=c_1\pi(\textbf{1})+c_2\pi(\textbf{2})+\cdots+c_n\pi(\textbf{n})\text{ for }\pi\in S_n,\text{ and }\\
&(\rho^{\text{def}}\otimes\rho^{\text{def}})(\pi)\left(v_1\otimes v_2\right)=\rho^{\text{def}}(\pi)(v_1)\otimes\rho^{\text{def}}(\pi)(v_2) \text{ for }\pi\in S_n,\; v_1\otimes v_2 \in \mathbb{C}\textbf{[n]}\otimes\mathbb{C}\textbf{[n]}.
\end{align*}
For a more general definition of the tensor product of two representations see \cite[Section 1.5]{Serre}. Let $\psi^{\lambda}$ denote the irreducible character \cite[Section 2.1]{Serre} of $S_n$ corresponding to the irreducible representation given by the partition $\lambda$ of $n$. If $\lambda\vdash n$ is non-self-conjugate, then we denote the irreducible character of $A_n$ corresponding to $\lambda$ by $\chi^{\lambda}$ and if $\lambda\vdash n$ is self-conjugate, then we denote the irreducible characters of $A_n$ corresponding to $\lambda$ by $\chi^{\lambda}_{\pm}$. Also, let $\varepsilon$ be the one-dimensional sign character of $S_n$. Let us recall that an $n$-dimensional linear representation $\rho$ over $\mathbb{C}$ of a finite group $G$ can also be defined by a homomorphism from $G$ to $M^*_n(\mathbb{C})$ (the multiplicative group of all $n\times n$ invertible complex matrices). We now define induced representation using this definition of representation.
\begin{defn}
Let $H$ be a subgroup of $G$ and $g_1H,\;g_2H,\;\dots,\;g_{\ell} H$ are all the distinct left cosets of $H$ in $G$. Let $\rho$ be a representation of $H$. The induced representation $\rho\uparrow_H^G$ is defined by
\begin{equation*}
\rho\uparrow_H^G(g)=\left(
{\begin{array}{cccc}
\rho(g_1^{-1}gg_1) & \rho(g_1^{-1}gg_2) & \dots & \rho(g_1^{-1}gg_{\ell})\\
\rho(g_2^{-1}gg_1) & \rho(g_2^{-1}gg_2) & \dots & \rho(g_2^{-1}gg_{\ell})\\
\vdots & \vdots & \ddots & \vdots\\
\rho(g_{\ell}^{-1}gg_1) & \rho(g_{\ell}^{-1}gg_2) & \dots & \rho(g_{\ell}^{-1}gg_{\ell})
\end{array}}
\right),
\end{equation*}
where $\rho(g)$ is the zero matrix if $g\notin H$. We denote the character of the induced representation $\rho\uparrow_H^G$ by $\chi\uparrow_H^G$, where $\chi$ is the character of $\rho$. Note that the dimension of the induced representation $\rho\uparrow_H^G$ is $[G:H]$dim$(\rho)$. We will abbreviate $\chi\uparrow_H^G$ to $\chi\uparrow^G$.
\end{defn}
Let us also recall that $\chi\downarrow_H^G$ denotes the restriction of the character $\chi$ to $H$, where $\chi$ is the character of a representation of the group $G$. We will abbreviate $\chi\downarrow_H^G$ to $\chi\downarrow_H$.
\begin{lem}\label{lem:Restriction}
If $\lambda$ be a non-self-conjugate partition of $n$, then $\psi^{\lambda}\downarrow_{A_n}=\chi^{\lambda}=\chi^{\lambda^{\prime}}$. If $\lambda$ is self-conjugate, then we have $\psi^{\lambda}\downarrow_{A_n}=\chi^{\lambda}_{+}+\chi^{\lambda}_{-}$.
\end{lem}\
\begin{proof}
The proof of this lemma follows directly from \cite[Theorem 4.6.5]{Amri}.
\end{proof}
\begin{prop}\label{prop:induced NC}
For non-self-conjugate $\lambda\vdash n$, we have $\chi^{\lambda}\uparrow^{S_n}=\psi^{\lambda}+\psi^{\lambda^{\prime}}$.
\end{prop}
\begin{proof}
$S_n$ can be written as the disjoint union of $A_n$ and $(1,2)A_n$ (two distinct left cosets of $A_n$ in $S_n$). Now the proposition follows from \cite[Theorem 4.4.2]{Amri} and definition of induced representation. 
\end{proof}
\begin{prop}\label{prop:induced Conj}
For self-conjugate $\lambda\vdash n$, we have $\chi_{\pm}^{\lambda}\uparrow^{S_n}=\psi^{\lambda}$.
\end{prop}
\begin{proof}
$A_n$ and $(1,2)A_n$ are two distinct left cosets of $A_n$ in $S_n$. Therefore the proposition follows from the definition of induced representation and  \cite[Theorem 4.4.2]{Amri}.
\end{proof}
\begin{prop}\label{prop:linearization of squared character}
Let us recall that $\chi^{\emph{def}}$ is the character of the defining representation $\rho^{\emph{def}}$ of $S_n$. Then
\[\left(\chi^{\emph{def}}\otimes\chi^{\emph{def}}\right)\downarrow_{A_n}=2\;\;\chi^{(n)}+3\;\;\chi^{(n-1,1)}+\;\chi^{(n-2,2)}+\;\chi^{(n-2,1,1)}.\]
\end{prop}
\begin{proof}
\cite[Example 2.1.8]{Sagan} and \cite[Lemma 2.9.16]{JK} gives,
\begin{equation}\label{eq:linearization for symm}
\left(\chi^{\text{def}}\otimes\chi^{\text{def}}\right)=2\;\;\psi^{(n)}+3\;\;\psi^{(n-1,1)}+\;\psi^{(n-2,2)}+\;\psi^{(n-2,1,1)}.
\end{equation}
For any non-self-conjugate $\lambda\vdash n$, Frobenius Reciprocity \cite[Theorem 1.12.6]{Sagan} implies
\begin{align}\label{eq: Linearization of character}
&\quad\left<\;\chi^{\lambda},\;\left(\chi^{\text{def}}\otimes\chi^{\text{def}}\right)\downarrow_{A_n}\right>\nonumber\\
&=\left<\;\chi^{\lambda}\uparrow^{S_n},\;\chi^{\text{def}}\otimes\chi^{\text{def}}\right>\nonumber\\
&=\left<\;\chi^{\lambda}\uparrow^{S_n},\;2\;\;\psi^{(n)}+3\;\;\psi^{(n-1,1)}+\;\psi^{(n-2,2)}+\;\psi^{(n-2,1,1)}\right>,\quad\text{ by \eqref{eq:linearization for symm}}\nonumber\\
&=\left<\;\psi^{\lambda}+\psi^{\lambda^{\prime}},\;2\;\;\psi^{(n)}+3\;\;\psi^{(n-1,1)}+\;\psi^{(n-2,2)}+\;\psi^{(n-2,1,1)}\right>,\quad\text{ by Proposition \ref{prop:induced NC}}.
\end{align}
Now using orthonormality of  irreducible characters of $S_n$, expression \eqref{eq: Linearization of character} becomes
\[
\begin{cases}
2,\text{ if }\lambda=(n)\text{ or }(1^n),\\
3,\text{ if }\lambda=(n-1,1)\text{ or }(2,1^{n-1}),\\
1,\text{ if }\lambda=(n-2,2)\text{ or }(2^2,1^{n-4}),\\
1,\text{ if }\lambda=(n-2,1,1)\text{ or }(3,1^{n-3}).
\end{cases}
\]
Again, for any self-conjugate $\lambda\vdash n$, by Frobenius Reciprocity, we have
\begin{align*}
&\left<\;\chi_{\pm}^{\lambda},\;\left(\chi^{\text{def}}\otimes\chi^{\text{def}}\right)\downarrow_{A_n}\right>\\
&=\left<\;\chi_{\pm}^{\lambda}\uparrow^{S_n},\;\chi^{\text{def}}\otimes\chi^{\text{def}}\right>\\
&=\left<\;\chi_{\pm}^{\lambda}\uparrow^{S_n},\;2\;\;\psi^{(n)}+3\;\;\psi^{(n-1,1)}+\;\psi^{(n-2,2)}+\;\psi^{(n-2,1,1)}\right>,\quad\text{ by \eqref{eq:linearization for symm}}\\
&=\left<\;\psi^{\lambda},\;2\;\;\psi^{(n)}+3\;\;\psi^{(n-1,1)}+\;\psi^{(n-2,2)}+\;\psi^{(n-2,1,1)}\right>,\quad\text{ by Proposition \ref{prop:induced Conj}}\\
&=0,\quad\text{ using orthonormality of irreducible characters of }S_n.
\end{align*}
Thus the proposition follows.
\end{proof}
\begin{lem}\label{lem: diagonal for defining action by A_n}
For any $i\in\{1,\dots,n\}$, the number of even permutations $\pi$ in $A_n$ which fixes $i\;($i.e.  $\pi(i)=i)$ is $\frac{1}{2}(n-1)!$.
\end{lem}
\begin{proof}
We know that the number of permutations $\pi\in S_n$ which fixes $i\in\{1,\dots,n\}$ is $(n-1)!$. Now for each $i$, let us consider the following sets,
\begin{equation*}
\mathcal{S}_i=\{\pi\in S_n\mid\pi(i)=i\}\;\text{ and }\;\mathcal{A}_i=\{\pi\in A_n\mid\pi(i)=i\}.
\end{equation*}
Then we can define a bijection $\psi:\mathcal{A}_i\rightarrow\mathcal{S}_i\setminus\mathcal{A}_i$ by $\pi\mapsto\pi(j,k)$, for fixed $j,k\in\{1,\dots,n\}\setminus\{i\}$ such that $j\neq k$. Therefore, the cardinality of $\mathcal{A}_i$ is same as the cardinality of  $\mathcal{S}_i\setminus\mathcal{A}_i$. But we also know that the cardinality of $\mathcal{S}_i$ is $(n-1)!$. Thus $|\mathcal{S}_i|=|\mathcal{A}_i|+|\mathcal{S}_i\setminus\mathcal{A}_i|$ implies $|\mathcal{A}_i|=\frac{1}{2}(n-1)!$. 
\end{proof}
Let us define a random variable $X$ on $A_n$ by $X(\pi):=$ the number of fixed points of $\pi$, $\pi\in A_n$. Therefore, we have $X(\pi)=\chi^{\text{def}}\downarrow_{A_n}(\pi)$. %Recall that $U_{A_n}$ is the uniform distribution on $A_n$ and let $P$ be the distribution defined in \eqref{def of P} on $A_n$. 
We now find the expectation $E_{U_{A_n}}(X)$ of $X$ under $U_{A_n}$.
\begin{equation}\label{eq: E_U_{A_n}}
E_{U_{A_n}}(X)=\sum\limits_{\pi\in A_n}X(\pi)U_{A_n}(\pi)=\sum\limits_{\pi\in A_n}\chi^{\text{def}}\downarrow_{A_n}(\pi)\frac{2}{n!}
=\frac{2}{n!}\sum\limits_{\pi\in A_n}\text{ Trace}\left(\rho^{\text{def}}\downarrow_{A_n}(\pi)\right).
\end{equation}
We know that for each $i\in \{1,\dots,n\}$, the $(i,i)^{\text{th}}$ entry of the matrix $\left(\sum\limits_{\pi\in A_n}\rho^{\text{def}}\downarrow_{A_n}(\pi)\right)$ is the number of  permutations in $A_n$ which fixes $i$. Therefore, using Lemma \ref{lem: diagonal for defining action by A_n} and the linearity of the trace, expression \eqref{eq: E_U_{A_n}} becomes
$\frac{2}{n!}\sum\limits_{i=1}^{n}\frac{(n-1)!}{2}=1.$

Now we find the expectation $E_k(X)$ and variance $V_k(X)$ of the same random variable $X$ under the distribution $P^{*k}$ on $A_n$. We know, by Example 2.1.8 and Theorem 2.11.2 of \cite{Sagan}, that the defining representation on $S_n$ decomposes into the trivial representation $S^{(n)}$ and the $(n-1)$-dimensional representation $S^{(n-1,1)}$. As the irreducible representations $S^{(n)}$ and $S^{(n-1,1)}$ of $S_n$ are irreducible in $A_n$ \cite[Theorem 4.6.7]{Amri}, we have the following:
% % % % % % %-----E_k(X)-----
\begin{align}\label{eq: E_k1}
E_k(X)=\sum\limits_{\pi\in A_n}X(\pi)P^{*k}(\pi)&=\sum\limits_{\pi\in A_n}\chi^{\text{def}}\downarrow_{A_n}(\pi)P^{*k}(\pi)\nonumber\\
&=\sum\limits_{\pi\in A_n}P^{*k}(\pi)\text{ Trace}\left(\rho^{\text{def}}\downarrow_{A_n}(\pi)\right).
\end{align}
Now using $\left(\sum\limits_{\pi\in A_n}P(\pi)\rho^{\text{def}}\downarrow_{A_n}(\pi)\right)^k=\sum\limits_{\pi\in A_n}P^{*k}(\pi)\rho^{\text{def}}\downarrow_{A_n}(\pi)$ and linearity of the trace, expression \eqref{eq: E_k1} is equal to Trace$\left(\widehat{P}\left(\rho^{\text{def}}\downarrow_{A_n} \right)\right)^k$. Therefore from Table \ref{Table}, we have
\begin{equation}\label{eq: E_k}
E_k(X)=1+(n-3)\left(\frac{2n-5}{2n-3}\right)^k+\left(\frac{n-3}{2n-3}\right)^k\left(1+(-1)^k\right),
\end{equation}
%-----TABLE-----
\begin{table}
	\begin{center}
		\begin{tabular}{c|c}
			\hline\\
			$\text{Partition of } n$ & $\text{ Eigenvalues of  }\widehat{P}(R)\text{ corresponding to the partition of column }1$ \\\\ \hline\hline \\
			$(n)$ or $(1^n)$ & $1$ \text{ with algebraic multiplicity} $1$ \\\\ \hline \\ 
			& $\frac{2n-5}{2n-3}$ \text{ with algebraic multiplicity} $n-3$ \\\\
			$(n-1,1)$ or $(2,1^{n-1})$ & $-\frac{n-3}{2n-3}$ \text{ with algebraic multiplicity} $1$  \\\\ 
			& $\frac{n-3}{2n-3}$ \text{ with algebraic multiplicity} $1$ \\\\  \hline \\
			&$\frac{2n-7}{2n-3}$ \text{ with algebraic multiplicity} $\frac{(n-2)(n-5)}{2}$\\\\
			&$\frac{-1}{2n-3}$ \text{ with algebraic multiplicity} $1$\\\\
			$(n-2,2)$ or $(2^2,1^{n-4})$ & $\frac{n-3}{2n-3}$ \text{ with algebraic multiplicity} $n-3$ \\\\ 
			& $-\frac{n-3}{2n-3}$ \text{ with algebraic multiplicity} $n-3$ \\\\ \hline \\
			&$\frac{2n-7}{2n-3}$ \text{ with algebraic multiplicity} $\frac{(n-3)(n-4)}{2}$\\\\
			&$\frac{3}{2n-3}$ \text{ with algebraic multiplicity} $1$\\\\
			$(n-2,1^2)$ or $(3,1^{n-3})$ & $\frac{n-5}{2n-3}$ \text{ with algebraic multiplicity} $n-3$ \\\\ 
			& $-\frac{n-5}{2n-3}$ \text{ with algebraic multiplicity} $n-3$ \\\\ \hline 
		\end{tabular}
	\end{center}
\caption{Eigenvalues of $\widehat{P}(R)$ corresponding to some specific irreducible representations of $A_n$}\label{Table}
\end{table}
% % % % % % %-----E_k(X^2)-----
\begin{align}\label{eq: V_k1}
E_k(X^2)=\sum\limits_{\pi\in A_n}\left(X(\pi)\right)^2P^{*k}(\pi)&=\sum\limits_{\pi\in A_n}\left(\chi^{\text{def}}\downarrow_{A_n}(\pi)\right)^2P^{*k}(\pi)\nonumber\\
&=\sum\limits_{\pi\in A_n}\left(\chi^{\text{def}}\otimes\chi^{\text{def}}\right)\downarrow_{A_n}(\pi)P^{*k}(\pi).
\end{align}
Using Proposition \ref{prop:linearization of squared character}, expression \eqref{eq: V_k1} can be written as,
\begin{equation}\label{eq:LB E_k^2}
\sum\limits_{\pi\in A_n}P^{*k}(\pi)\left(2\;\;\chi^{(n)}(\pi)+3\;\;\chi^{(n-1,1)}(\pi)+\;\chi^{(n-2,2)}(\pi)+\;\chi^{(n-2,1,1)}(\pi)\right).
\end{equation}
Now if we write $\xi=\sum\limits_{\pi\in A_n}P(\pi)\;\pi$, then using the definition of character and linearity of the trace, expression \eqref{eq:LB E_k^2} equal to,
\begin{align*}
& 2\;\;\chi^{(n)}(\xi^k)+3\;\;\chi^{(n-1,1)}(\xi^k)+\;\chi^{(n-2,2)}(\xi^k)+\;\chi^{(n-2,1,1)}(\xi^k)\\
=&2+3\left((n-3)\left(\frac{2n-5}{2n-3}\right)^k+\left(\frac{n-3}{2n-3}\right)^k\left(1+(-1)^k\right)\right)\\
&\;+\left(\frac{(n-2)(n-5)}{2}\left(\frac{2n-7}{2n-3}\right)^k+\left(\frac{-1}{2n-3}\right)^k+(n-3)\left(\frac{n-3}{2n-3}\right)^k\left(1+(-1)^k\right)\right)\\
&\;+\left(\frac{(n-3)(n-4)}{2}\left(\frac{2n-7}{2n-3}\right)^k+\left(\frac{3}{2n-3}\right)^k+(n-3)\left(\frac{n-5}{2n-3}\right)^k\left(1+(-1)^k\right)\right)
\end{align*}
from Table \ref{Table}.

%------------V_k-------
Therefore from the definition of variance  $V_k(X)=E_k(X^2)-\left(E_k(X)\right)^2$ we have 
\begin{align}\label{eq: V_k}
V_k(X)= \;& 1+(n-3)\left(\frac{2n-5}{2n-3}\right)^k+\left(\frac{n-3}{2n-3}\right)^k\left(1+(-1)^k\right)\\
&\;\;-\left((n-3)\left(\frac{2n-5}{2n-3}\right)^k+\left(\frac{n-3}{2n-3}\right)^k\left(1+(-1)^k\right)\right)^2\nonumber\\
&\quad+\frac{(n-2)(n-5)}{2}\left(\frac{2n-7}{2n-3}\right)^k+\left(\frac{-1}{2n-3}\right)^k+(n-3)\left(\frac{n-3}{2n-3}\right)^k\left(1+(-1)^k\right)\nonumber\\
&\quad+\frac{(n-3)(n-4)}{2}\left(\frac{2n-7}{2n-3}\right)^k+\left(\frac{3}{2n-3}\right)^k+(n-3)\left(\frac{n-5}{2n-3}\right)^k\left(1+(-1)^k\right).\nonumber
\end{align}
%---Appox. V_k and E_k---
\begin{lem}\label{thm:appx E_k and V_k}
For $V_k(X)$ and $E_k(X)$ defined in \eqref{eq: E_k} and \eqref{eq: V_k} respectively, the following are true
\begin{enumerate}
\item For $c< 0,\;k=(n-\frac{3}{2})(\log n+c)$ and large $n$, we have
    \begin{align*}& E_k(X)\approx 1+e^{-c}\left(1+o(1)\right)\\& V_k(X)\approx 1+e^{-c}\left(1+o(1)\right).\end{align*}
\item For any $\epsilon\in(0,1),\;k_n=\lfloor(1-\epsilon)(n-\frac{3}{2})\log n\rfloor$ and large $n$, we have
    \begin{align*}& E_{k_n}(X)\approx 1+n^{\epsilon}+o(1)\\& V_{k_n}(X)\approx 1+n^{\epsilon}+o(1)+n^{\epsilon}o(1),\end{align*}
\end{enumerate}
where `$\approx$' means `asymptotic to' i.e. $a_n\approx b_n$ means $\lim\limits_{n\rightarrow\infty}\frac{a_n}{b_n}=1$.
\end{lem}
\begin{proof}
Throughout this proof we denote $d_k=\left(\frac{1+(-1)^k}{2^k}\right)$ for convenience. Note that $0\leq d_k< 1$ for $k\geq 1$. First, we have
\begin{equation}\label{eq: E_k appx in terms of k}
\begin{split}
E_k(X)&=1+(n-3)\left(1-\frac{1}{n-1.5}\right)^k+\left(1-\frac{1.5}{n-1.5}\right)^k\left(\frac{1+(-1)^k}{2^k}\right)\\
&\approx 1+(n-3)e^{-\frac{k}{n-1.5}}+d_ke^{-\frac{1.5k}{n-1.5}}.
\end{split}
\end{equation}
Now from \eqref{eq: V_k}, we have
\begin{align}\label{eq: V_k appx in terms of k}
V_k(X)\approx\;& 1+(n-3)e^{-\frac{k}{n-1.5}}+d_ke^{-\frac{1.5k}{n-1.5}}-\left((n-3)e^{-\frac{k}{n-1.5}}+d_ke^{-\frac{1.5k}{n-1.5}}\right)^2\\
&\quad\quad\quad\quad+\frac{(n-2)(n-5)}{2}e^{-\frac{2k}{n-1.5}}+\left(\frac{-1}{2n-3}\right)^k+(n-3)d_ke^{-\frac{1.5k}{n-1.5}}\nonumber\\
&\quad\quad\quad\quad+\frac{(n-3)(n-4)}{2}e^{-\frac{2k}{n-1.5}}+\left(\frac{3}{2n-3}\right)^k+(n-3)d_ke^{-\frac{3.5k}{n-1.5}}\nonumber\\
=\;& 1+(n-3)e^{-\frac{k}{n-1.5}}+d_ke^{-\frac{1.5k}{n-1.5}}-(n-2)e^{-\frac{2k}{n-1.5}}-2d_k(n-3)e^{-\frac{2.5k}{n-1.5}}\nonumber\\
&\;\;-d_k^2e^{-\frac{3k}{n-1.5}}+\left(\frac{-1}{2n-3}\right)^k+(n-3)d_ke^{-\frac{1.5k}{n-1.5}}+\left(\frac{3}{2n-3}\right)^k+(n-3)d_ke^{-\frac{3.5k}{n-1.5}}.\nonumber
\end{align}
Now for large $n$, if we take $k=(n-\frac{3}{2})(\log n+c),\;c< 0$, then we have $k\geq 1$ and hence $d_k$ is bounded above. Therefore from \eqref{eq: E_k appx in terms of k}, we have $E_k(X)\approx1+e^{-c}\left(1+o(1)\right)$ and from \eqref{eq: V_k appx in terms of k}, we have $V_k(X)\approx1+e^{-c}\left(1+o(1)\right)$ by straightforward calculations. This proves the first part.

Now for any $\epsilon\in(0,1),\;k_n=\lfloor(1-\epsilon)(n-\frac{3}{2})\log n\rfloor$ and we have $k_n\geq 1$ for large $n$. Hence $d_{k_n}$ is bounded above. Therefore from \eqref{eq: E_k appx in terms of k}, we have $E_{k_n}(X)\approx1+n^{\epsilon}+o(1)$ and from \eqref{eq: V_k appx in terms of k}, we have $V_{k_n}(X)\approx1+n^{\epsilon}+o(1)+n^{\epsilon}o(1)$ by straightforward calculations. This proves the second part.
\end{proof}
\begin{thm}\label{thm:LB for Alt}
We have the following:
\begin{enumerate}
\item For large $n,\;||P^{*k}-U_{A_n}||_{\text{TV}}\geq1-\frac{6}{1+e^{-c}(1+o(1))}$, when $k=(n-\frac{3}{2})(\log n+c)$ and $c\ll0$.
\item $\lim\limits_{n\rightarrow\infty}||P^{*k_n}-U_{A_n}||_{\text{TV}}=1$, for any $\epsilon\in (0,1)$ and $k_n=\lfloor(1-\epsilon)(n-\frac{3}{2})\log n\rfloor.$
\end{enumerate}
\end{thm}
\begin{proof} 
For any positive constant $a$, by Chebychev's inequality, we have
\begin{equation}\label{eq: LB CSV}
P^{*k}\left(\{\pi\in A_n: |X(\pi)-E_k(X)|\leq a\sqrt{V_k(X)} \}\right)\geq1-\frac{1}{a^2}.
\end{equation}
Now we choose positive constant $a$ such that $E_k(X)-a\sqrt{V_k(X)}>0$. Then by Markov's inequality we have,
\begin{align}\label{eq: LB MRKV}
U_{A_n}\left(\{\pi\in A_n: X(\pi)\geq E_k(X)-a\sqrt{V_k(X)} \}\right) & \leq  \frac{E_{U_{A_n}}(X)}{E_k(X)-a\sqrt{V_k(X)}}\\
&=\frac{1}{E_k(X)-a\sqrt{V_k(X)}}\nonumber.
\end{align}
Now from the definition of total variation distance, we have
\begin{align}\label{eq:LB}
||P^{*k}-U_{A_n}||_{\text{TV}}&=\sup_{A\subset A_n}|P^{*k}(A)-U_{A_n}(A)|\nonumber\\
&\geq P^{*k}\left(\{\pi\in A_n: |X(\pi)-E_k(X)|\leq a\sqrt{V_k(X)} \}\right)\nonumber\\
&\quad-U_{A_n}\left(\{\pi\in A_n: |X(\pi)-E_k(X)|\leq a\sqrt{V_k(X)} \}\right)\nonumber\\
% % % % % % % % % % % % %--------------------------------------------------------- % % % % % % % % % % %
&\geq P^{*k}\left(\{\pi\in A_n: |X(\pi)-E_k(X)|\leq a\sqrt{V_k(X)} \}\right)\nonumber\\
&\quad-U_{A_n}\left(\{\pi\in A_n: X(\pi)\geq E_k(X)-a\sqrt{V_k(X)} \}\right)\nonumber\\
% % % % % % % % % % % % %-----------------------------------------------------------------
&\geq 1-\frac{1}{a^2}-\frac{1}{E_k(X)-a\sqrt{V_k(X)}}.
\end{align}
The inequality \eqref{eq:LB} follows by using \eqref{eq: LB CSV} and \eqref{eq: LB MRKV}. In particular, if we take $a=\frac{E_k(X)}{2\sqrt{V_k(X)}}>0$ in the above inequality, we get
\begin{equation}\label{eq:LB main ineq}
||P^{*k}-U_{A_n}||_{\text{TV}}\geq 1-\frac{4V_k(X)}{(E_k(X))^2}-\frac{2}{E_k(X)}.
\end{equation}
Now if $n$ is large, $c\ll0$ and $k=(n-\frac{3}{2})(\log n+c)$, then by \eqref{eq:LB main ineq} and by the first part of Lemma \ref{thm:appx E_k and V_k}, we have the first part of this theorem.

Again for any $\epsilon\in (0,1)$ and $k_n=\lfloor(1-\epsilon)(n-\frac{3}{2})\log n\rfloor$ from \eqref{eq:LB main ineq} and by the second part of Lemma \ref{thm:appx E_k and V_k}, we have the following:
\begin{equation}\label{eq:LB cutoff limit}
1\geq ||P^{*k_n}-U_{A_n}||_{\text{TV}}\geq  1-\frac{4(1+n^{\epsilon}+o(1)+n^{\epsilon}o(1))}{(1+n^{\epsilon}+o(1))^2}-\frac{2}{1+n^{\epsilon}+o(1)},
\end{equation}
for large $n$. Therefore, the second part of this theorem follows from \eqref{eq:LB cutoff limit} and the following:
\begin{align*}
&\lim\limits_{n\rightarrow\infty}\frac{4(1+n^{\epsilon}+o(1)+n^{\epsilon}o(1))}{(1+n^{\epsilon}+o(1))^2}=0,\\
&\lim\limits_{n\rightarrow\infty}\frac{2}{1+n^{\epsilon}+o(1)}=0.
\end{align*}
\end{proof}
\begin{cor}[Total variation cutoff for transpose top-$2$ with random shuffle]
The transpose top-$2$ with random shuffle on $A_n$ driven by $P\;$ %( defined in \eqref{def of P}) 
exhibits the total variation cutoff phenomenon and the cutoff is at $(n-\frac{3}{2})\log n$.
\end{cor}
\begin{proof}
The proof follows from the second part of the Theorems \ref{thm:UB for Alt} and \ref{thm:LB for Alt}.
\end{proof}
%------------Example--------
For example, if $n=10$, the plot for $||P^{*k}-U_{A_{10}}||_{\text{TV}}$ vs. $k$ is given in Figure \ref{Figure: Fig}. In this case the cutoff is at $8.5\times \log 10 = 19.572$ and Figure \ref{Figure: Fig} confirms this.
%---------Figure----------
\begin{figure}[h]
	\centering
	\includegraphics{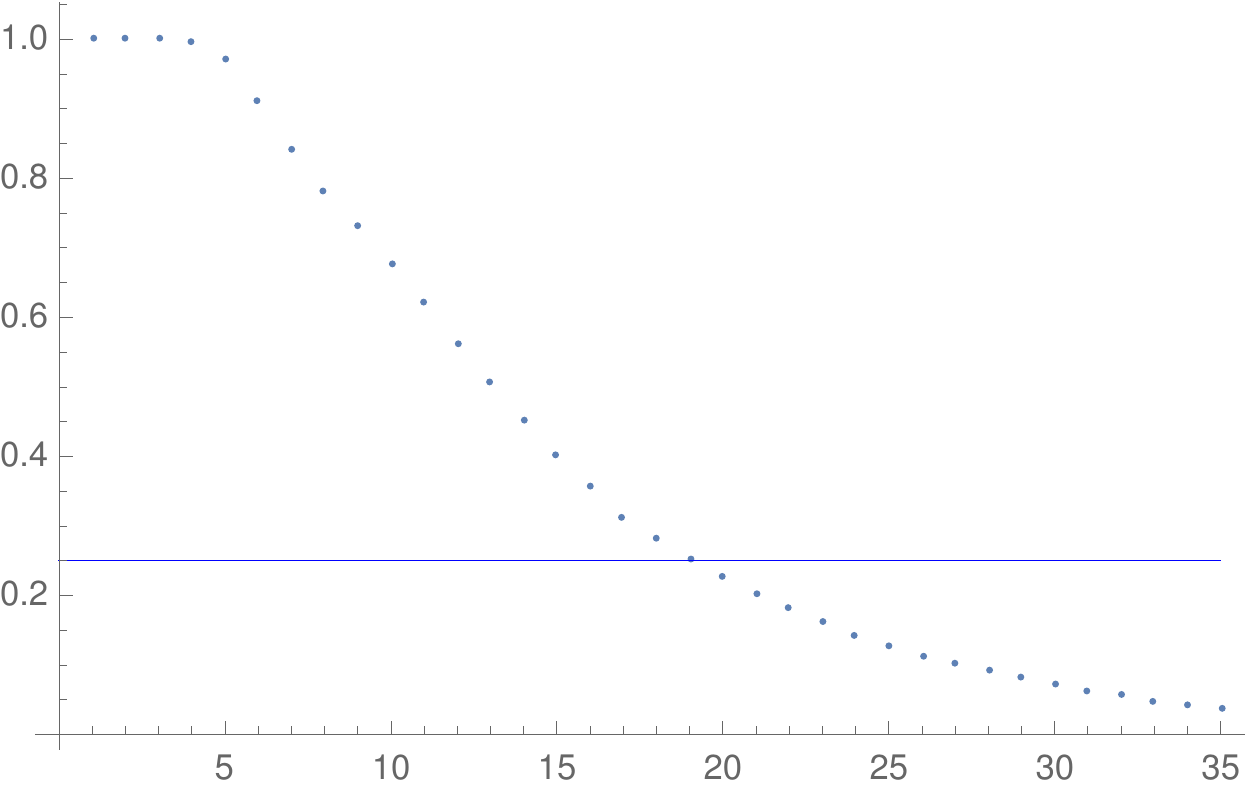}
	\caption{ Plot for $||P^{*k}-U_{A_{10}}||_{\text{TV}}$ vs. $k$ }
	\label{Figure: Fig}
\end{figure}
%----------REFERENCES---------

% % % % % % % % % % % % % % % % % % % % % % % % % % % % % % % % % % % % % % % % % % % % % % % % % % % % % % % % % % % % % % % % % % % %
\end{document}